\documentclass[12pt]{article}
\usepackage{e-jc}
\usepackage{amsthm,amsmath,amssymb}

\usepackage{mathrsfs}
\usepackage{amsfonts}
\usepackage[colorlinks=true,citecolor=black,linkcolor=black,urlcolor=blue]{hyperref}
\usepackage{amsrefs}
\usepackage[all]{xy}
\everymath{\displaystyle}

\title{\bf On Blow-Ups and Injectivity of Quivers}

\author{Will Grilliette\\
\small Department of Mathematics\\[-0.8ex]
\small Tyler Junior College\\[-0.8ex] 
\small Tyler, TX\\
\small\tt w.b.grilliette@gmail.com\\
\and
Deborah E.\ Seacrest\qquad Tyler Seacrest\\
\small Department of Mathematics\\[-0.8ex]
\small University of Montana Western\\[-0.8ex] 
\small Dillon, MT\\
\small\tt \{d\_seacrest,t\_seacrest\}@umwestern.edu\\
}

\date{\dateline{Oct 3, 2012}{Feb 10, 2013}}

\newtheorem{thm}{Theorem}[section]
\newtheorem{prop}[thm]{Proposition}
\newtheorem{cor}[thm]{Corollary}
\newtheorem{lem}[thm]{Lemma}

\theoremstyle{definition}
\newtheorem*{defn}{Definition}

\theoremstyle{remark}
\newtheorem{ex}[thm]{Example}

\numberwithin{figure}{section}

\DeclareMathOperator{\edges}{edges}

\newcommand{\Set}{\mathbf{Set}}
\newcommand{\Quiv}{\mathbf{Quiv}}
\newcommand{\cat}[1]{\mathscr{#1}}
\DeclareMathOperator{\Ran}{ran}
\DeclareMathOperator{\ob}{Ob}

\begin{document}

\maketitle
\begin{abstract}
This work connects the idea of a ``blow-up" of a quiver with that of injectivity, showing that for a class of monic maps $\Phi$, a quiver is $\Phi$-injective if and only if all blow-ups of it are as well.  This relationship is then used to characterize all quivers that are injective with respect to the natural embedding of $P_{n}$ into $C_{n}$.
\end{abstract}

\section{Introduction}
This paper continues the study of injectivity for quivers begun in \cite{grilliette3}.  Recall that a quiver is a directed multigraph with loops, and a quiver homomorphism is a pair of maps preserving the adjacency structure.  Explicitly, the definitions below will be used throughout the paper, where $\Set$ is the category of sets with functions.

\begin{defn}[Quiver, {\cite[Definition 2.1]{graphtransformation}}]
A \emph{quiver} is a quadruple $(V,E,\sigma,\tau)$, where ${V,E\in\ob(\Set)}$ are sets, and $\sigma,\tau\in\Set(E,V)$ are functions.  Elements of $V$ are \emph{vertices}, and $V$ the \emph{vertex set}.  Elements of $E$ are \emph{edges}, and $E$ the \emph{edge set}.  The function $\sigma$ is the \emph{source map}, and $\tau$ the \emph{target map}.  For $e\in E$, $\sigma(e)$ is the \emph{source} of $e$, and $\tau(e)$ the \emph{target} of $e$.
\end{defn}

\begin{defn}[Quiver map, {\cite[Definition 2.4]{graphtransformation}}]
Given quivers $G$ and $H$, a \emph{quiver homomorphism} from $G$ to $H$ is a pair $\left(\phi_{V},\phi_{E}\right)$, where $\phi_{V}\in\Set\left(V_{G},V_{H}\right)$ and $\phi_{E}\in\Set\left(E_{G},E_{H}\right)$ satisfy $\phi_{V}\circ\sigma_{G}=\sigma_{H}\circ\phi_{E}$ and $\phi_{V}\circ\tau_{G}=\tau_{H}\circ\phi_{E}$.  The function $\phi_{V}$ is the \emph{vertex map}, and $\phi_{E}$ the \emph{edge map}.
\end{defn}

Let $\Quiv$ denote the category of quivers in which the composition of quiver homomorphisms is defined component-wise.  The notation of this paper will follow that established in \cite{grilliette3}.  Please also observe that to ease notation, the functors $V$ and $E$ will be omitted on maps, as is convention in graph theory.  In \cite{grilliette3}, the class of quivers which are injective with respect to all monic quiver maps was characterized using the following generalization of a complete graph.

\begin{defn}[{\cite[Definition 3.1.1]{grilliette3}}]
For a quiver $J$ and $v,w\in V(J)$, let
\[
\edges_{J}(v,w):=\sigma_{J}^{-1}(v)\cap\tau_{J}^{-1}(w)
\]
be the set of all edges in $J$ with source $v$ and target $w$.  A quiver $J$ is \emph{loaded} if for every $v,w\in V(J)$, $\edges_{J}(v,w)\neq\emptyset$.
\end{defn}

Essentially, a loaded quiver is a complete digraph, which may have multiple edges.  From \cite[Proposition 3.2.1]{grilliette3}, a quiver is injective with respect to all monic quiver maps if and only if it is loaded and has a nonempty vertex set.

The notion of a ``blow-up'' of a quiver, analogous to the concept for simple graphs in \cite{hatami2011}, is introduced and shown to be intimately tied to quiver injectivity in Theorem \ref{converse}.  That is, if a quiver is injective, all of its blow-ups are also injective.  As injectivity classes are closed on retraction, an injective quiver yields a family of injectives:  its retracts and their blow-ups.  Section \ref{retraction-section} characterizes quiver sections and retractions, which are used in Section \ref{blowup-section} to show that a blow-up retracts onto the original quiver and yield the main result in Theorem \ref{converse}.

The remainder of the paper is dedicated to an example of using blow-ups to characterize a class of injective objects.  For $n\geq 2$, let $\xymatrix{P_{n}\ar[r]^{\phi_{n}} & C_{n}}$ be the natural embedding of the directed path $P_{n}$ into the directed cycle $C_{n}$.  Propositions \ref{ncyclic-union} and \ref{nowalk} and Theorem \ref{structure} characterize a $\phi_{n}$-injective quiver as a quiver such that the connected components are either blow-ups of $C_{m}$ for $m\mid n$ or quivers with no walk of length $n-1$.  This seems to show that injectivity yields a propagation of structure, as observed in Lemmas \ref{propagation}, \ref{triangulation}, and \ref{loaded}.

Studying quivers injective relative to specific maps seems to be an interesting way in which a local condition propagates to give rise to global structure.  In the case of embedding $P_n$ into $C_n$, this process yields the surprisingly simple structure of blow-ups of $C_n$.  The case of embedding $P_3$ into a transitive triangle in the natural way gives rise to ``transitive digraphs" \cite[\S 4.3]{bang2008}, which correspond to directed sets.   It may be that embeddings not yet studied give rise to other interesting structures.

Because injectivity is a categorical notion, any structure resulting from injectivity can be compared to an algebraic counterpart.  Since the notion of a retract is vital to the proof of Theorems \ref{converse} and \ref{structure}, this is also one of the very few times categorical methods have been used in graph theory to prove combinatorial results.

The authors would like to thank the referees of this paper for their comments and patience in its revision.

\section{Sections and Retractions in $\Quiv$}\label{retraction-section}

Recall from \cite[Definition 7.19]{joyofcats} that a \emph{section} in a category is a morphism that is left-invertible.  Also, \cite[Definition 7.24]{joyofcats} gives the dual notion of a section as a \emph{retraction}, a right-invertible morphism.  Traditionally, the codomain of a retraction, or equivalently the domain of a section, is termed a \emph{retract}.

For quivers, a section can be recognized by finding a certain type of partition of the codomain's vertices and edges.

\begin{prop}[Characterization of sections]\label{section}
A homomorphism $\xymatrix{G\ar[r]^{j} & H}\in\Quiv$ is a section if and only if there are partitions $\left(A_{v}\right)_{v\in V(G)}$ of $V(H)$ and $\left(B_{e}\right)_{e\in E(G)}$ of $E(H)$ satisfying
\begin{enumerate}
\item $j(v)\in A_{v},$
\item $j(e)\in B_{e},$
\item $\sigma_{H}(f)\in A_{\sigma_{G}(e)}$,
\item $\tau_{H}(f)\in A_{\tau_{G}(e)}$,
\end{enumerate}
for all $v\in V(G)$, $e\in E(G)$, and $f\in B_{e}$.  In this case, $G$ is a retract of $H$.
\end{prop}

\begin{proof}

$\left(\Rightarrow\right)$ Let $\xymatrix{H\ar[r]^{q} & G}\in\Quiv$ satisfy that $q\circ j=id_{G}$.  For $v\in V(G)$ and $e\in E(H)$, define $A_{v}:=q^{-1}(v)$ and $B_{e}:=q^{-1}(e)$, which partition $V(H)$ and $E(H)$.  Notice that
\[\begin{array}{ccc}
v=(q\circ j)(v)=q\left(j(v)\right)	&	\textrm{and}	&	e=(q\circ j)(e)=q\left(j(e)\right)\\
\end{array}\]
for all $v\in V(G)$ and $e\in E(G)$.  Thus, $j(v)\in A_{v}$ and $j(e)\in B_{e}$.  Also,
\[
q\left(\sigma_{H}(f)\right)=\sigma_{G}\left(q(f)\right)=\sigma_{G}(e)
\]
and
\[
q\left(\tau_{H}(f)\right)=\tau_{G}\left(q(f)\right)=\tau_{G}(e)
\]
for all $e\in V(G)$ and $f\in B_{e}$.  Thus, $\sigma_{H}(f)\in A_{\sigma_{G}(e)}$ and $\tau_{H}(f)\in A_{\tau_{G}(e)}$

$\left(\Leftarrow\right)$ Define $\xymatrix{V(H)\ar[r]^{q_{V}} & V(G)}\in\Set$ and $\xymatrix{E(H)\ar[r]^{q_{E}} & E(G)}\in\Set$ by
\[\begin{array}{ccc}
q_{V}(w):=v	&	\textrm{and}	&	q_{E}(f):=e,
\end{array}\]
where $w\in A_{v}$ and $f\in B_{e}$.  Both are well-defined, as $\left(A_{v}\right)_{v\in V(G)}$ and $\left(B_{e}\right)_{e\in E(G)}$ are partitions.  For $f\in E(H)$, there is $e\in E(G)$ such that $f\in B_{e}$.  Thus,
\[
\sigma_{G}\left(q_{E}(f)\right)
=\sigma_{G}(e)
=q_{V}\left(\sigma_{H}(f)\right),
\]
as $\sigma_{H}(f)\in A_{\sigma_{G}(e)}$.  Likewise, $\tau_{G}\left(q_{E}(f)\right)=q_{V}\left(\tau_{H}(f)\right)$ by a similar argument using ${\tau_{H}(f)\in A_{\tau_{G}(e)}}$.  Thus, $q:=\left(q_{V},q_{E}\right)$ is a quiver homomorphism.  Moreover,
\[\begin{array}{ccc}
(q\circ j)(v)=q_{V}\left(j(v)\right)=v	&	\textrm{and}	&	(q\circ j)(e)=q_{E}\left(j(e)\right)=e\\
\end{array}\]
as $j(v)\in A_{v}$ and $j(e)\in B_{e}$ for all $v\in V(G)$ and $e\in E(G)$.  Thus, $q\circ j=id_{G}$.

\end{proof}

Dually, a retraction can be recognized by finding pre-images of the vertices and edges that align with the structure maps.

\begin{prop}[Characterization of retractions]\label{retraction}
A homomorphism $\xymatrix{H\ar[r]^{q} & G}\in\Quiv$ is a retraction if and only if there are $\left(w_{v}\right)_{v\in V(G)}\subseteq V(H)$ and $\left(f_{e}\right)_{e\in E(G)}\subseteq E(H)$ satisfying
\begin{enumerate}
\item $q\left(w_{v}\right)=v,$
\item $q\left(f_{e}\right)=e,$
\item $\sigma_{H}\left(f_{e}\right)=w_{\sigma_{G}(e)}$,
\item $\tau_{H}\left(f_{e}\right)=w_{\tau_{G}(e)}$,
\end{enumerate}
for all $v\in V(G)$ and $e\in E(G)$.  In this case, $G$ is a retract of $H$.
\end{prop}

\begin{proof}

$\left(\Rightarrow\right)$ Let $\xymatrix{G\ar[r]^{j} & H}\in\Quiv$ satisfy that $q\circ j=id_{G}$.  Define $w_{v}:=j(v)$ and $f_{e}:=j(e)$ for $v\in V(G)$ and $e\in E(G)$.  Observe that
\[
v=(q\circ j)(v)=q\left(j(v)\right)=q\left(w_{v}\right)
\]
and
\[
e=(q\circ j)(e)=q\left(j(e)\right)=q\left(f_{e}\right)
\]
for all $v\in V(G)$ and $e\in E(G)$.  Also,
\[
\sigma_{H}\left(f_{e}\right)
=\left(\sigma_{H}\circ j\right)(e)
=\left(j\circ\sigma_{G}\right)(e)
=w_{\sigma_{G}(e)}
\]
and
\[
\tau_{H}\left(f_{e}\right)
=\left(\tau_{H}\circ j\right)(e)
=\left(j\circ\tau_{G}\right)(e)
=w_{\tau_{G}(e)}
\]
for all $e\in E(G)$.

$\left(\Leftarrow\right)$ Define $\xymatrix{V(G)\ar[r]^{j_{V}} & V(H)}\in\Set$ and $\xymatrix{E(G)\ar[r]^{j_{E}} & E(H)}\in\Set$ by
\[\begin{array}{ccc}
j_{V}(v):=w_{v}	&	\textrm{and}	&	j_{E}(e):=f_{e}.\\
\end{array}\]
For $e\in E(G)$,
\[
\left(\sigma_{H}\circ j_{E}\right)(e)
=\sigma_{H}\left(f_{e}\right)
=w_{\sigma_{G}(e)}
=\left(j_{V}\circ\sigma_{G}\right)(e)
\]
and likewise $\left(\tau_{H}\circ j_{E}\right)(e)=\left(j_{V}\circ\tau_{G}\right)(e)$, using $\tau_{H}\left(f_{e}\right)=w_{\tau_{G}(e)}$.  Thus, $j:=\left(j_{V},j_{E}\right)$ is a quiver homomorphism.  Moreover,
\[\begin{array}{ccc}
(q\circ j)(v)=q\left(w_{v}\right)=v	&	\textrm{and}	&	(q\circ j)(e)=q\left(f_{e}\right)=e\\
\end{array}\]
for all $v\in V(G)$ and $e\in E(G)$.  Thus, $q\circ j=id_{G}$.

\end{proof}

For any category, sections are monic and retractions epic.  Due to the characterization of quiver monomorphisms, a copy of a retract can always be found in the codomain quiver of a section, or dually the domain quiver of a retraction, as illustrated below.

\begin{ex}\label{badness}
The following maps are a section and retraction pair, where the primes are associated to their counterparts.
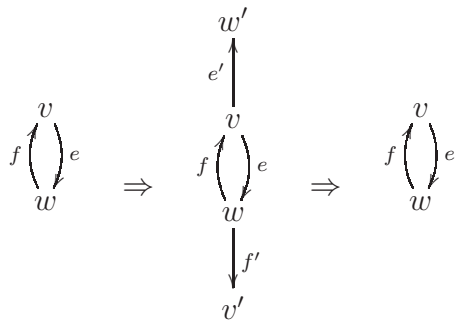
\[\begin{array}{ccccc}
\xymatrix{
\\
v\ar@/^/[d]^{e}\\
w\ar@/^/[u]^{f}\\
}
&
\xymatrix{\\ \\ \Rightarrow \\}
&
\xymatrix{
w'\\
v\ar@/^/[d]^{e}\ar[u]^{e'}\\
w\ar@/^/[u]^{f}\ar[d]^{f'}\\
v'\\
}
&
\xymatrix{\\ \\ \Rightarrow \\}
&
\xymatrix{
\\
v\ar@/^/[d]^{e}\\
w\ar@/^/[u]^{f}\\
}
\end{array}\]
\end{ex}

\section{Blow-up of a Quiver}\label{blowup-section}

The notion of a ``blow-up'' of a simple graph $G$ is already defined in \cite[p.\ 1]{hatami2011} as a new graph constructed by replacing $v\in V(G)$ with a set of vertices $A_{v}$ and defining $x\in A_{v}$ and $y\in A_{w}$ adjacent if and only if $v$ and $w$ were adjacent in $G$.  The introduction of multiple edges and direction obscures this intuitive idea, but the blow-up of a quiver follows the same concept.

For a quiver $G$, each vertex $v\in V(G)$ is replaced with a set of vertices $A_{v}$ and each edge $e$ with a set of edges $B_{e}$, satisfying appropriate adjacency conditions.  If $\xymatrix{v\ar[r]^{e} & w}$ is an edge in $G$ and $f\in B_{e}$, the source and target of $f$ must be in $A_{v}$ and $A_{w}$, respectively.  Conversely, if $x\in A_{v}$ and $y\in B_{w}$, there must be an edge from $x$ to $y$ in the blow-up associated to $e$.

\begin{defn}
Given a quiver $G$, a \emph{blow-up} of $G$ is a quiver $H$ equipped with partitions $\left(A_{v}\right)_{v\in V(G)}$ of $V(H)$ and $\left(B_{e}\right)_{e\in E(G)}$ of $E(H)$ satisfying
\begin{enumerate}
\item $\sigma_{H}(f)\in A_{\sigma_{G}(e)}$,
\item $\tau_{H}(f)\in A_{\tau_{G}(e)}$,
\item $\edges_{H}(x,y)\cap B_{e}\neq\emptyset$,
\end{enumerate}
for all $e\in E(G)$, $f\in B_{e}$, $x\in A_{\sigma_{G}(e)}$, $y\in A_{\tau_{G}(e)}$.
\end{defn}

\begin{ex}[Reflexivity]\label{reflexive}
Any quiver $G$ is a blow-up of itself by using $A_{v}:=\{v\}$ and $B_{e}:=\{e\}$ for $v\in V(G)$ and $e\in E(G)$.
\end{ex}

\begin{ex}[Loaded]\label{loadedbouquets}
The blow-ups of the single loop bouquet $C_{1}$ are precisely all nonempty loaded quivers.
\end{ex}

\begin{ex}[Cycles]
For $n\geq 2$, the blow-ups of the directed cycle $C_{n}$ take the form below, where the bold arrows indicate the existence of at least one arrow between any vertex of the source set and any vertex of the target set.  Note that since $C_{n}$ has no loops, the sets $A_{j}$ are independent.
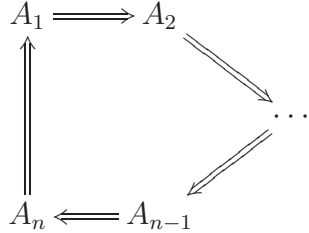
\[\xymatrix{
A_{1}\ar@{=>}[r]	&	A_{2}\ar@{=>}[rd]\\
&	&	\cdots\ar@{=>}[ld]\\
A_{n}\ar@{=>}[uu]	&	A_{n-1}\ar@{=>}[l]\\
}\]
Consequently, the middle quiver in Example \ref{badness} is not a blow-up of $C_{2}$.
\end{ex}

From this definition, any blow-up of a quiver can be retracted back onto the original.  This can be done by constructing a section using Proposition \ref{section}, or by constructing a retraction using Proposition \ref{retraction}.

\begin{prop}\label{blowup}
Given a quiver $G$ and a blow-up $H$ of $G$, $G$ is a retract of $H$.
\end{prop}

\begin{proof}

Let $\left(A_{v}\right)_{v\in V(G)}$ and $\left(B_{e}\right)_{e\in E(G)}$ be the partitions of $V(H)$ and $E(H)$, respectively.  Define $\xymatrix{V(H)\ar[r]^{q_{V}} & V(G)}\in\Set$ by $q_{V}(w):=v$, where $w\in A_{v}$.  Likewise, define\\ ${\xymatrix{E(H)\ar[r]^{q_{E}} & E(G)}\in\Set}$ by $q_{E}(f):=e$, where $f\in B_{e}$.  Both are well-defined, as $\left(A_{v}\right)_{v\in V(G)}$ and $\left(B_{e}\right)_{e\in E(G)}$ are partitions.  For $f\in E(H)$, there is $e\in E(G)$ such that $f\in B_{e}$.  Thus,
\[
\left(\sigma_{G}\circ q_{E}\right)(f)
=\sigma_{G}(e)
=\left(q_{V}\circ\sigma_{H}\right)(f)
\]
and
\[
\left(\tau_{G}\circ q_{E}\right)(f)
=\tau_{G}(e)
=\left(q_{V}\circ\tau_{H}\right)(f)
\]
as $\sigma_{H}(f)\in A_{\sigma_{G}(e)}$ and $\tau_{H}(f)\in A_{\tau_{G}(e)}$.  Hence, $q:=\left(q_{V},q_{E}\right)$ is a quiver homomorphism.  For $v\in V(G)$, choose $w_{v}\in A_{v}$.  For $e\in E(G)$, choose $f_{e}\in\edges_{H}\left(w_{\sigma_{G}(e)},w_{\tau_{G}(e)}\right)\cap B_{e}$.  Then,
\[\begin{array}{ccc}
q\left(w_{v}\right)=v,	&	q\left(f_{e}\right)=e,	&	\sigma_{H}\left(f_{e}\right)=w_{\sigma_{G}(e)},
\end{array}\]
and $\tau_{H}\left(f_{e}\right)=w_{\tau_{G}(e)}$.  By Proposition \ref{retraction}, $q$ is a retraction, and $G$ is a retract of $H$.

\end{proof}

In a general category, retracts have a pleasant relationship with injectivity from \cite[\S 9]{joyofcats}.  Given a category $\cat{C}$ and a class of morphisms $\Phi$, an object $A$ of $\cat{C}$ is $\Phi$-injective if and only if all retracts of $A$ are.  In the case of $\cat{C}=\Quiv$, if a blow-up of a quiver is $\Phi$-injective, then by Proposition \ref{blowup}, the original quiver must also be $\Phi$-injective.  Moreover, the converse of this statement about blow-ups is also true if $\Phi$ is a class of monomorphisms.

\begin{thm}\label{converse}
Let $G$ be a quiver and $\Phi$ a class of monic quiver maps.  The following are equivalent:
\begin{enumerate}
\item $G$ is $\Phi$-injective,
\item all blow-ups of $G$ are $\Phi$-injective,
\item a blow-up of $G$ is $\Phi$-injective.
\end{enumerate}
\end{thm}

\begin{proof}

$(2\Rightarrow 3)$ If all blow-ups of $G$ are $\Phi$-injective, then a single blow-up is.

$(3\Rightarrow 1)$ By Proposition \ref{blowup}, $G$ is a retract of any blow-up, so $G$ becomes $\Phi$-injective by the general theory.

$(1\Rightarrow 2)$ Let $H$ be a blow-up of $G$, equipped with partitions $\left(A_{v}\right)_{v\in V(G)}$ and $\left(B_{e}\right)_{e\in E(G)}$ of $V(H)$ and $E(H)$, respectively.  Given $\xymatrix{D\ar[r]^{\phi} & C}\in\Phi$ and $\xymatrix{D\ar[r]^{\psi} & H}\in\Quiv$, the goal is to construct a map $\xymatrix{C\ar[r]^{\tilde{\psi}} & H}\in\Quiv$ such that $\psi=\tilde{\psi}\circ\phi$.  Let $\xymatrix{H\ar[r]^{q} & G}\in\Quiv$ be the retraction constructed in Proposition \ref{blowup}.
\[\xymatrix{
H\ar@{->>}[r]^{q}	&	G\\
D\textrm{ }\ar@{>->}[r]_{\phi}\ar[u]^{\psi}	&	C\\
}\]
As $G$ is $\Phi$-injective, there is $\xymatrix{C\ar[r]^{\hat{\psi}} & G}\in\Quiv$ such that $q\circ\psi=\hat{\psi}\circ\phi$.
\[\xymatrix{
H\ar@{->>}[r]^{q}	&	G\\
D\textrm{ }\ar@{>->}[r]_{\phi}\ar[u]^{\psi}	&	C\ar@{..>}[u]_{\exists\hat{\psi}}\\
}\]

To construct the quiver map from $C$ to $H$, vertices and edges will be chosen much like in Proposition \ref{blowup}.  First, the vertices of $C$ will be handled.  For $v\in\Ran(\phi)$, there is a unique $a_{v}\in V(D)$ such that $v=\phi\left(a_{v}\right),$ since $\phi$ is monic.  Let $w_{v}:=\psi\left(a_{v}\right)$ and observe that
\[
q\left(w_{v}\right)
=(q\circ\psi)\left(a_{v}\right)
=\left(\hat{\psi}\circ\phi\right)\left(a_{v}\right)
=\hat{\psi}(v).
\]
Thus, $w_{v}\in A_{\hat{\psi}(v)}$.  For $v\not\in\Ran(\phi)$, choose $w_{v}\in A_{\hat{\psi}(v)}$ arbitrarily.

Next, consider the edges of $C$.  For $e\in\Ran(\phi)$, there is a unique $b_{e}\in E(D)$ such that $e=\phi\left(b_{e}\right),$ since $\phi$ is monic.  Let $f_{e}:=\psi\left(b_{e}\right)$ and observe that
\[
q\left(f_{e}\right)
=(q\circ\psi)\left(b_{e}\right)
=\left(\hat{\psi}\circ\phi\right)\left(b_{e}\right)
=\hat{\psi}(e),
\]
\[
\sigma_{H}\left(f_{e}\right)
=\left(\sigma_{H}\circ \psi\right)\left(b_{e}\right)
=\left(\psi\circ\sigma_{D}\right)\left(b_{e}\right)
=\psi\left(a_{\sigma_{C}(e)}\right)
=w_{\sigma_{C}(e)},
\]
and $\tau_{H}\left(f_{e}\right)=w_{\tau_{C}(e)}$ by a similar calculation since $\phi$ is monic.  Thus,\\ ${f_{e}\in\edges_{H}\left(w_{\sigma_{C}(e)},w_{\tau_{C}(e)}\right)\cap B_{\hat{\psi}(e)}}$.  For $e\not\in\Ran(\phi)$, choose\\ ${f_{e}\in\edges_{H}\left(w_{\sigma_{C}(e)},w_{\tau_{C}(e)}\right)\cap B_{\hat{\psi}(e)}}$ arbitrarily.

Define $\xymatrix{V(C)\ar[r]^{\alpha} & V(H)}\in\Set$ by $\alpha(v):=w_{v}$ and $\xymatrix{E(C)\ar[r]^{\beta} & E(H)}\in\Set$ by $\beta(e):=f_{e}$.  For $e\in E(C)$,
\[
\left(\sigma_{H}\circ\beta\right)(e)
=\sigma_{H}\left(f_{e}\right)
=w_{\sigma_{C}(e)}
=\left(\alpha\circ\sigma_{C}\right)(e)
\]
and
\[
\left(\tau_{H}\circ\beta\right)(e)
=\tau_{H}\left(f_{e}\right)
=w_{\tau_{C}(e)}
=\left(\alpha\circ\tau_{C}\right)(e)
\]
Thus, $\tilde{\psi}:=\left(\alpha,\beta\right)$ is a quiver homomorphism.
\[\xymatrix{
H\ar@{->>}[r]^{q}	&	G\\
D\textrm{ }\ar@{>->}[r]_{\phi}\ar[u]^{\psi}	&	C\ar@{->}[u]_{\hat{\psi}}\ar@{..>}[ul]_{\tilde{\psi}}\\
}\]
Notice that $\tilde{\psi}\circ\phi=\psi$ by design.  Since $\phi$ was arbitrary, $H$ is $\Phi$-injective.

\end{proof}

\begin{ex}
Notice that the bouquet $C_{1}$ is a terminal object in $\Quiv$.  Consequently, it is injective with respect to any class of quiver morphisms, in particular any class $\Phi$ of monomorphisms.  By Example \ref{loadedbouquets} and Theorem \ref{converse}, any nonempty loaded quiver is $\Phi$-injective.  The content of \cite[Proposition 3.2.1]{grilliette3} is that when $\Phi$ is the class of all monic maps, the injectivity class contains no other members.
\end{ex}

Similarly, abstract projectivity classes are closed under retractions, but the dual statement of Theorem \ref{converse} is not true.

\begin{ex}\label{projectivefail}
By \cite[Proposition 4.1.1]{grilliette3}, the directed path of order-2, $P_{2}$, is projective with respect to all quiver epimorphisms.  However, the following quiver is a blow-up of $P_{2}$, but not epi-projective.
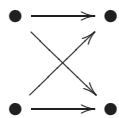
\[\xymatrix{
\bullet\ar[r]\ar[dr]	&	\bullet\\
\bullet\ar[r]\ar[ur]	&	\bullet\\
}\]
\end{ex}

\section{Definition and Abstract Relationships}

The remainder of this paper will consider some particular cases of quiver injectivity, where the concept of a blow-up will become quite telling and useful.

For an integer $n\geq 2$, let $P_{n}$ be the directed path of order $n$ and $C_{n}$ the directed cycle of order $n$.  In this section, $P_{n}$ will be embedded into $C_{n}$ by connecting the two endpoints of the path in the obvious way.  This embedding will be denoted as the quiver monomorphism $\phi_{n}:P_{n}\to C_{n}$.  This situation is drawn below.
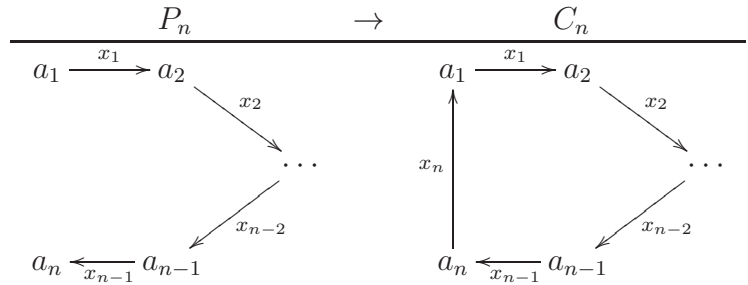
\[\begin{array}{ccc}
P_{n}	&	\to	&	C_{n}\\
\hline
\xymatrix{a_{1}\ar[r]^{x_{1}} & a_{2}\ar[dr]^{x_{2}}\\ & & \cdots\ar[dl]^{x_{n-2}}\\ a_{n} & a_{n-1}\ar[l]^{x_{n-1}}\\ }
&	&
\xymatrix{a_{1}\ar[r]^{x_{1}} & a_{2}\ar[dr]^{x_{2}}\\ & & \cdots\ar[dl]^{x_{n-2}}\\ a_{n}\ar[uu]^{x_{n}} & a_{n-1}\ar[l]^{x_{n-1}}\\ }\\
\end{array}\]

The goal will be to describe all quivers which are $\phi_{n}$-injective.  Much like \cite[Proposition 3.2.1]{grilliette3}, the property of being $\phi_{n}$-injective can be recovered with only a condition on the quiver itself.

\begin{prop}[Characterization of $\phi_{n}$-injectivity]\label{ncyclic}
A quiver $J$ is $\phi_{n}$-injective if and only if for any walk of length $n-1$, there is an edge from the terminal vertex to the initial vertex.
\end{prop}

\begin{proof}
($\Rightarrow$) Let $\left(v_{j}\right)_{j=1}^{n}\subseteq V(J)$ with $e_{j}\in\edges_{J}\left(v_{j},v_{j+1}\right)$ for $1\leq j\leq n-1$.  Define $\psi_{V}:V\left(P_{n}\right)\to V(J)$ and $\psi_{E}:E\left(P_{n}\right)\to E(J)$ by
\[\begin{array}{ccc}
a_{j}\mapsto v_{j}
&	\textrm{and}	&
x_{k}\mapsto e_{k},\\
\end{array}\]
for all $1\leq j\leq n$ and $1\leq k\leq n-1$.  A routine check shows that $\psi:=\left(\psi_{V},\psi_{E}\right)$ is a quiver map from $P_{n}$ to $J$.  As $J$ is $\phi_{n}$-injective, there is a quiver map $\hat{\psi}:C_{n}\to J$ such that $\hat{\psi}\circ\phi_{n}=\psi$.  Let $e_{n}:=\hat{\psi}\left(x_{n}\right)$.  A calculation shows that $\sigma_{J}\left(e_{n}\right)=v_{n}$ and $\tau_{J}\left(e_{n}\right)=v_{1}$.

($\Leftarrow$) Given $\psi:P_{n}\to J$, let
\[\begin{array}{ccc}
v_{j}:=\psi\left(a_{j}\right)
&	\textrm{and}	&
e_{k}:=\psi\left(x_{k}\right),\\
\end{array}\]
for all $1\leq j\leq n$ and $1\leq k\leq n-1$.  A calculation shows that $e_{j}\in\edges_{J}\left(v_{j},v_{j+1}\right)$ for $1\leq j\leq n-1$.  By assumption, there is $e_{n}\in\edges_{J}\left(v_{n},v_{1}\right)$.  Define $\hat{\psi}_{V}:V\left(C_{n}\right)\to V(J)$ and $\hat{\psi}_{E}:E\left(C_{n}\right)\to E(J)$ by
\[\begin{array}{ccc}
a_{j}\mapsto v_{j}
&	\textrm{and}	&
x_{j}\mapsto e_{j},\\
\end{array}\]
for all $1\leq j\leq n$.  A routine check shows that $\hat{\psi}:=\left(\hat{\psi}_{V},\hat{\psi}_{E}\right)$ is a quiver map from $C_{n}$ to $J$ and that $\hat{\psi}\circ\phi=\psi$.
\end{proof}

\begin{ex}[$\phi_{n}$-injectivity of some quivers] Using this criterion, some quivers can immediately be shown to be $\phi_{n}$-injective or not.
\begin{enumerate}
\item Any nonempty loaded quiver is mono-injective by \cite[Proposition 3.2.1]{grilliette3} and, therefore, trivially $\phi_{n}$-injective for all $n\geq 2$.
\item For $n,m\geq 2$, $P_{m}$ is $\phi_{n}$-injective if and only if $n>m$.
\item For any $n,m\geq 2$, $C_{m}$ is $\phi_{n}$-injective if and only if $m\mid n$.
\end{enumerate}
\end{ex}

By \cite[Proposition 10.40]{joyofcats}, abstract injectivity classes are closed on the categorical product, which is no less true here.  In most algebraic settings, a product object is injective if and only if all of its factors are.  However, this is not the case for $\phi_{n}$-injectivity.

\begin{ex}[Failure of products for $\phi_{n}$-injectivity]
Observe that $P_{2}\prod P_{m}$ is $\phi_{n}$-injective for $n\geq 3$ and $m\geq 2$.
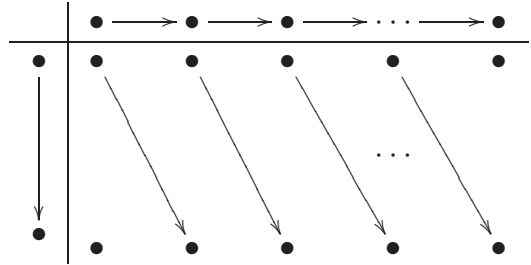
\[\begin{array}{c|c}
&	\xymatrix{\bullet\ar[r]	&	\bullet\ar[r]	&	\bullet\ar[r]	&	\cdots\ar[r]	&	\bullet}\\
\hline
\xymatrix{\bullet\ar[dd]\\ \\ \bullet}	&	\xymatrix{\bullet\ar[ddr]	&	\bullet\ar[ddr]	&	\bullet\ar[ddr]	&	 \bullet\ar[ddr]	&	\bullet\\
&	&	&	\cdots\\
\bullet	&	\bullet	&	\bullet	&	\bullet	&	\bullet\\}\\
\end{array}\]
Let $I_{m}$ denote the independent set of vertices of order-$m$.  Then, $I_{1}\prod P_{m}\cong I_{m}$ is $\phi_{2}$-injective.
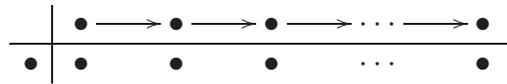
\[\begin{array}{c|c}
&	\xymatrix{\bullet\ar[r]	&	\bullet\ar[r]	&	\bullet\ar[r]	&	\cdots\ar[r]	&	\bullet}\\
\hline
\bullet	&	\xymatrix{\bullet	&	\bullet	&	\bullet	&	\cdots	&	\bullet}\\
\end{array}\]
\end{ex}

On the other hand, these particular injectivity classes have a very tight relationship with the coproduct, the disjoint union, which can be proven quickly with Proposition \ref{ncyclic}.  This result is due to the weak connectivity of $P_{n}$ and $C_{n}$.

\begin{prop}[Disjoint unions and $\phi_{n}$-injectivity]\label{ncyclic-union}
Let $\Lambda$ be an index set and $J_{\lambda}$ a quiver for each $\lambda\in\Lambda$.  Then, $\coprod_{\lambda\in\Lambda}J_{\lambda}$ is $\phi_{n}$-injective if and only if $J_{\lambda}$ is $\phi_{n}$-injective for each $\lambda\in\Lambda$.
\end{prop}

\begin{proof}
Let $J:=\coprod_{\lambda\in\Lambda}J_{\lambda}$, and regard each $J_{\lambda}$ as a subquiver of $J$.

($\Rightarrow$) Fix $\lambda\in\Lambda$.  Let $\left(v_{j}\right)_{j=1}^{n}\subseteq V\left(J_{\lambda}\right)$ with $e_{j}\in\edges_{J_{\lambda}}\left(v_{j},v_{j+1}\right)$ for $1\leq j\leq n-1$.  Since $J$ is $\phi_{n}$-injective and $J_{\lambda}$ is a subquiver of $J$, there must be $e_{n}\in\edges_{J}\left(v_{n},v_{1}\right)$ by Proposition \ref{ncyclic}.  As a disjoint union, the only edges in $J$ between $v_{n}$ and $v_{1}$ arise from $J_{\lambda}$, forcing $e_{n}\in\edges_{J_{\lambda}}\left(v_{n},v_{1}\right)$.  Proposition \ref{ncyclic} states that $J_{\lambda}$ must be $\phi_{n}$-injective.

($\Leftarrow$) Let $\left(v_{j}\right)_{j=1}^{n}\subseteq V(J)$ with $e_{j}\in\edges_{J}\left(v_{j},v_{j+1}\right)$ for $1\leq j\leq n-1$.  As a disjoint union, $e_{1}$ arises from some $J_{\lambda}$, forcing that $v_{1}$ and $v_{2}$ must also be from the same $J_{\lambda}$.  By induction, $e_{k}$ and $v_{j}$ must be from $J_{\lambda}$ also for all $1\leq j\leq n$ and $1\leq k\leq n-1$.  Since $J_{\lambda}$ is $\phi_{n}$-injective, there is $e_{n}\in\edges_{J_{\lambda}}\left(v_{n},v_{1}\right)$, which also exists in $J$.
\end{proof}

By this result, one need only consider the weakly connected components of a quiver to determine $\phi_{n}$-injectivity.

\section{Trivial Cases}

Looking at the examples thus far, most quivers have been $\phi_{n}$-injective for trivial reasons, either being loaded or having no walk of length $n-1$.  The latter property can be used to completely describe the structure of the quiver.

\begin{prop}[No walk of length $n$]\label{nowalk}
If $J$ is a quiver with no walk of length $n$, then $J$ consists of $n-1$ sets of vertices $B_{1}, \ldots, B_{n-1}$, where every edge of $J$ goes from a vertex in $B_{i}$ to a vertex in $B_{j}$ for some $i < j$.
\end{prop}

\begin{proof}
Let $B_1$ be the set of all sources of $J$.   If $J$ has no walk of length $n$, then $J$ has no cycles, and hence $B_1$ is non-empty if $J$ is non-empty.  For every vertex $v \notin B_1$, place $v \in B_i$ if $i$ is the maximum length of a path ending with $v$.

Hence all that needs to be shown is that every edge goes from a vertex $B_i$ to a vertex $B_j$ for $i < j$.  Suppose there is an edge from $u \in B_i$ to $v \in B_j$.   By definition, there is a path of length $i$ ending on $u$.  By taking the edge $u \to v$, this means there is a path of length $i+1$ ending on $v$.  Since $j$ is the length of the longest path ending on $v$, it must be that $j \geq i+1 > i$, as desired.
\end{proof}

For $n\geq 3$, the loaded case can be identified readily by one characteristic feature, having a loop at some vertex.  This fact will be proven incrementally over three lemmas.  First, an edge terminating in a loop spawns a copy of the full quiver on two vertices.

\begin{lem}[Propagating loops]\label{propagation}
For $n\geq 3$, suppose $J$ is a $\phi_{n}$-injective quiver.  If $J$ has a walk of length 2, which has distinct endpoints and either starts or ends with a loop, then there is an edge in reverse and a loop on the opposite vertex.
\end{lem}

\begin{proof}
Say the following is a subquiver of $J$.
\[\xymatrix{
u\ar[r]^{g}	&	w\ar@(ur,dr)^{f}\\
}\]
Then, consider the walk given by
\[\begin{array}{ccc}
\begin{array}{rcl}
v_{1}	&	:=	&	u,\\
v_{j}	&	:=	&	w,\\
\end{array}
&	\textrm{and}	&
\begin{array}{rcl}
e_{1}	&	:=	&	g,\\
e_{k}	&	:=	&	f,\\
\end{array}
\end{array}\]
for $2\leq j\leq n$ and $2\leq k\leq n-1$.  By Proposition \ref{ncyclic}, there is $h\in\edges_{J}(w,u)$.
\[\xymatrix{
u\ar@/^/[r]^{g}	&	w\ar@(ur,dr)^{f}\ar@{..>}@/^/[l]^{h}\\
}\]
Now, consider the walk given by
\[\begin{array}{ccc}
\begin{array}{rcl}
v_{1}:=v_{n}	&	:=	&	u,\\
v_{j}	&	:=	&	w,\\
\end{array}
&	\textrm{and}	&
\begin{array}{rcl}
e_{1}	&	:=	&	g,\\
e_{k}	&	:=	&	f,\\
e_{n-1}	&	:=	&	h,\\
\end{array}
\end{array}\]
for $2\leq j\leq n-1$ and $2\leq k\leq n-2$.  By Proposition \ref{ncyclic}, there is $d\in\edges_{J}(u,u)$.
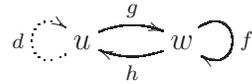
\[\xymatrix{
u\ar@/^/[r]^{g}\ar@{..>}@(dl,ul)^{d}	&	w\ar@(ur,dr)^{f}\ar@/^/[l]^{h}\\
}\]
If $g$ was reversed, the subquiver appears between $u$ and $w$ by dualizing the proof.

\end{proof}

Next, two consecutive edges give rise to a copy of the full quiver on three vertices, allowing the quiver to be triangulated.

\begin{lem}[Triangulation]\label{triangulation}
For $n\geq 3$, suppose $J$ is a $\phi_{n}$-injective quiver.  If $J$ has two consecutive edges ending with a loop, then an edge exists between each pair of vertices involved.
\end{lem}

\begin{proof}

Say the following subquiver exists in $J$,
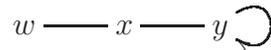
\[\xymatrix{
w\ar@{-}[r]	&	x\ar@{-}[r]	&	y\ar@(ur,dr)\\
}\]
where the two edges can be oriented in either direction.  By Lemma \ref{propagation}, each edge becomes a 2-cycle with a loop at each vertex.
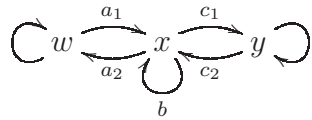
\[\xymatrix{
\\
w\ar@(dl,ul)\ar@/^/[r]^{a_{1}}	&	x\ar@(dr,dl)^{b}\ar@/^/[r]^{c_{1}}\ar@/^/[l]^{a_{2}}	&	 y\ar@(ur,dr)\ar@/^/[l]^{c_{2}}\\
\\
}\]
Consider the walk given by
\[\begin{array}{ccc}
\begin{array}{rcl}
v_{1}	&	:=	&	w,\\
v_{j}	&	:=	&	x,\\
v_{n}	&	:=	&	y,\\
\end{array}
&	\textrm{and}	&
\begin{array}{rcl}
e_{1}	&	:=	&	a_{1},\\
e_{k}	&	:=	&	b,\\
e_{n-1}	&	:=	&	c_{1},\\
\end{array}
\end{array}\]
for $2\leq j\leq n-1$ and $2\leq k\leq n-2$.  By Proposition \ref{ncyclic}, there is $d_{1}\in\edges_{J}(y,w)$.\\
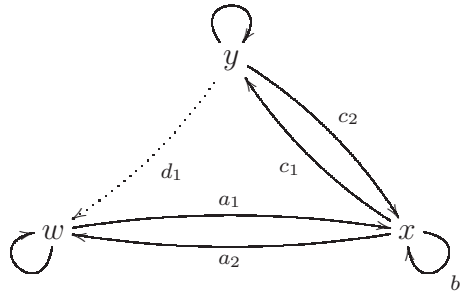
\[\xymatrix{
&	&	y\ar@(ul,ur)\ar@/^/[ddrr]^{c_{2}}\ar@{..>}@/^/[ddll]^{d_{1}}\\
\\
w\ar@(d,l)\ar@/^/[rrrr]^{a_{1}}	&	&	&	&	x\ar@(r,d)^{b}\ar@/^/[uull]^{c_{1}}\ar@/^/[llll]^{a_{2}}\\
\\
}\]
By a symmetric argument, there is $d_{2}\in\edges_{J}(w,y)$.

\end{proof}

At last, these two lemmas together yield the following characterization of when a $\phi_{n}$-injective quiver is loaded.

\begin{lem}[$\phi_{n}$-injective components, loaded]\label{loaded}
For $n\geq 3$, suppose $J$ is a $\phi_{n}$-injective quiver with $V(J)\neq\emptyset$.  Then, $J$ is loaded if and only if it is both weakly connected and has at least one vertex with a loop.
\end{lem}

\begin{proof}

$\left(\Rightarrow\right)$ This follows from \cite[Definition 3.1.1]{grilliette3} and \cite[Proposition 3.2.1]{grilliette3}.

$\left(\Leftarrow\right)$ Let $w\in V(J)$ such that there is $f\in\edges_{J}(w,w)$.  Consider arbitrary ${x,y\in V(J)}$.  Since $J$ is weakly connected, there is some finite sequence of edges from $x$ to $w$ and some finite sequence of edges from $w$ to $y$.
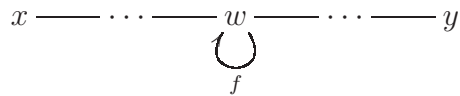
\[\xymatrix{
\\
x\ar@{-}[r]	&	\cdots\ar@{-}[r]	&	w\ar@(dr,dl)^{f}\ar@{-}[r]	&	\cdots\ar@{-}[r]	&	y\\
\\
}\]
By induction on the length of the edge sequence from $x$ to $w$ to $y$, Lemma \ref{propagation} yields the following.
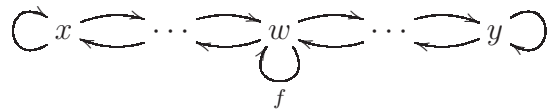
\[\xymatrix{
\\
x\ar@(dl,ul)\ar@/^/[r]	&	\cdots\ar@/^/[r]\ar@/^/[l]	&	w\ar@(dr,dl)^{f}\ar@/^/[r]\ar@/^/[l]	&	 \cdots\ar@/^/[r]\ar@/^/[l]	&	y\ar@(ur,dr)\ar@/^/[l]\\
\\
}\]
Induction on the length of the path with Lemma \ref{triangulation} then gives the following.
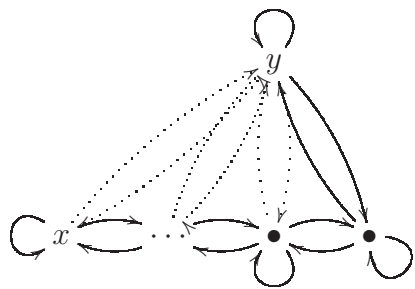
\[\xymatrix{
\\
&	&	y\ar@(ur,ul)\ar@/^/[ddr]\ar@{..>}@/^/[dd]\ar@{..>}@/^/[ddl]\ar@{..>}@/^/[ddll]\\
\\
x\ar@(ul,dl)\ar@{..>}@/^/[uurr]\ar@/^/[r]	&	\cdots\ar@/^/[r]\ar@/^/[l]\ar@{..>}@/^/[uur]	&	 \bullet\ar@(dr,dl)\ar@/^/[r]\ar@/^/[l]\ar@{..>}@/^/[uu]	&	\bullet\ar@(r,d)\ar@/^/[l]\ar@/^/[uul]\\
\\
}\]

Since $x$ and $y$ were arbitrary, $J$ is loaded.

\end{proof}

This is a very useful criterion since this reduces avoiding loaded quivers to avoiding loops.  However, this criterion does not hold for $n=2$ as shown in the example below.

\begin{ex}[A $\phi_{2}$-injective quiver]\label{thebadcase}
The following quiver is $\phi_{2}$-injective, but not $\phi_{n}$-injective for $n\geq 3$.
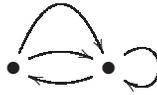
\[\xymatrix{
\\
\bullet\ar@/^2pc/[r]\ar@/^/[r]	&	\bullet\ar@(ur,dr)\ar@/^/[l]\\
\\
}\]
\end{ex}

However, $\phi_{2}$-injective quivers are readily understood from Proposition \ref{ncyclic}.

\begin{cor}[Structure of $\phi_{2}$-injective quivers]
A quiver $J$ is $\phi_{2}$-injective if and only if given any non-loop edge in $J$, there is an edge in reverse.
\end{cor}

From this, the bi-directing of any undirected multigraph is naturally $\phi_{2}$-injective, but notice that Example \ref{thebadcase} is not the bi-directing of an undirected multigraph.

\section{Blow-ups of Cycles}

For $n\geq 3$, what remains to be described is the structure of quivers that are $\phi_{n}$-injective and weakly connected with a walk of length $n-1$ and no loops.  The archetype of this case is $C_{m}$, where $m\mid n$.  By Theorem \ref{converse}, all blow-ups of $C_{m}$ are $\phi_{n}$-injective for $m\mid n$.  Moreover, these blow-ups exhaust this final case of $\phi_{n}$-injectivity.

\begin{thm}[Structural characterization of $\phi_n$-injective quivers]\label{structure}
For $n\geq 3$, consider a $\phi_{n}$-injective and weakly connected quiver $J$.  If $J$ is has a walk of length $n-1$, then $J$ is a blow-up of $C_{\ell}$ for some $\ell\mid n$.
\end{thm}

\begin{proof}

By hypothesis, there is a walk $\left(w_{j}\right)_{j=1}^{n}$ in $J$.  By Proposition \ref{ncyclic}, there is an edge from $w_{n}$ to $w_{1}$, meaning that $J$ has a closed walk.  Let
\[
\ell:=\min\{m\in\mathbb{N}:J\textrm{ has a closed walk of length }m\}.
\]
If $\ell=1$, Lemma \ref{loaded} states that $J$ is loaded since $J$ is weakly connected.  By Example \ref{loadedbouquets}, $J$ is a blow-up of $C_{1}$.

If $\ell\neq 1$, let $\left(v_{j}\right)_{j=0}^{\ell-1}$ and $e_{j}\in\edges_{J}\left(v_{j},v_{j+1\pmod\ell}\right)$ for $0\leq j\leq \ell-1$ be a closed walk of minimal length.  Consequently, the vertices are distinct as a repeated vertex yields a shorter closed walk.  Since $\ell\leq n$, $n=q\ell+r$ for some $q\in\mathbb{N}$ and $0\leq r<\ell$.  If $r\neq 0$, the walk $w_{k}:=v_{k\pmod\ell}$ for $0\leq k\leq n-1$ necessitates an edge from $v_{r-1\pmod\ell}$ to $v_{0}$.  This forms a closed walk of length $r$, contradicting the minimality of $\ell$.  Thus, $n=q\ell$.

Define $B_{j+1 \pmod \ell}:=\tau_{J}\left(\sigma_{J}^{-1}\left(v_{j}\right)\right)$ for $0\leq j\leq\ell-1$, the proposed partite sets for the blow-up of $C_{\ell}$.  For $x\in B_{0}$, there is $f\in\edges_{J}\left(v_{\ell-1},x\right)$.  The walk
\[\begin{array}{ccc}
w_{k}:=\left\{\begin{array}{cc}
v_{k\pmod\ell},	&	1\leq k\leq n-1,\\
x,	&	k=n,\\
\end{array}\right.
&	\textrm{and}	&
f_{k}:=\left\{\begin{array}{cc}
e_{k\pmod\ell},	&	1\leq k\leq n-2,\\
f,	&	k=n-1,\\
\end{array}\right.
\end{array}\]
forces an edge from $x$ to $v_{1}$.  Translation of this argument gives an edge from every element of $B_{j}$ to $v_{j+1\pmod\ell}$ for all $0\leq j\leq\ell-1$.

For $x\in B_{0}$ and $y\in B_{1}$, there is $f\in\edges_{J}\left(v_{\ell-1},x\right)$ and $g\in\edges_{J}\left(y,v_{2}\right)$.  The walk
\[\begin{array}{ccc}
w_{k}:=\left\{\begin{array}{cc}
y,	&	k=1,\\
v_{k\pmod\ell},	&	2\leq k\leq n-1,\\
x,	&	k=n,\\
\end{array}\right.
&	\textrm{and}	&
f_{k}:=\left\{\begin{array}{cc}
g,	&	k=1,\\
e_{k\pmod\ell},	&	2\leq k\leq n-2,\\
f,	&	k=n-1,\\
\end{array}\right.
\end{array}\]
forces an edge from $x$ to $y$.  Translation of this argument gives an edge from every element of $B_{j}$ to every element of $B_{j+1\pmod\ell}$ for all $0\leq j\leq\ell-1,$ as desired.

For $j\neq1$, assume that $x\in B_{0}$, $y\in B_{j}$, and $h\in\edges_{J}(x,y)$.  There are\\ ${f\in\edges_{J}\left(y,v_{j+1}\right)}$, $g\in\edges_{J}\left(v_{\ell-1},x\right)$, and $b\in\edges_{J}\left(x,v_{1}\right)$.  If $j>1$, the walk
\[\begin{array}{ccc}
w_{k}:=\left\{\begin{array}{cc}
y,	&	k=1,\\
v_{j+k-1},	&	2\leq k\leq\ell-j,\\
x,	&	k=\ell-j+1,\\
\end{array}\right.
&	\textrm{and}	&
f_{k}:=\left\{\begin{array}{cc}
f,	&	k=1,\\
e_{j+k-1},	&	2\leq k\leq \ell-j-1,\\
g,	&	k=\ell-j,\\
h,	&	k=\ell-j+1,\\
\end{array}\right.
\end{array}\]
becomes a closed walk of length $\ell-j+1$, contradicting the minimality of $\ell$.  If $j=0$, the walk
\[\begin{array}{ccc}
w_{k}:=\left\{\begin{array}{cc}
x,	&	k=1,\\
y,	&	k=2,\\
v_{k-2\pmod\ell},	&	3\leq k\leq n,\\
\end{array}\right.
&	\textrm{and}	&
f_{k}:=\left\{\begin{array}{cc}
h,	&	k=1,\\
f,	&	k=2,\\
e_{k-2\pmod\ell},	&	3\leq k\leq n-1,\\
\end{array}\right.
\end{array}\]
yields $a\in\edges_{J}\left(v_{\ell-2},x\right)$, making the closed walk of length $\ell-1$ below.
\[\begin{array}{ccc}
w_{k}:=\left\{\begin{array}{cc}
v_{k},	&	1\leq k\leq\ell-2,\\
x,	&	k=\ell-1,\\
\end{array}\right.
&	\textrm{and}	&
f_{k}:=\left\{\begin{array}{cc}
e_{k},	&	1\leq k\leq\ell-3,\\
a,	&	k=\ell-2,\\
b,	&	k=\ell-1.\\
\end{array}\right.
\end{array}\]
This again contradicts the minimality of $\ell$.  Therefore, edges from $B_{0}$ terminate in $B_{j}$ only when $j=1$.  Translation of this argument shows that edges from $B_{k}$ terminate in $B_{j}$ only when $j-k\equiv 1\pmod\ell$ as desired.

Assume that $\bigcup_{j=0}^{\ell-1}B_{j}\neq V(J)$.  For $z\not\in\bigcup_{j=0}^{\ell-1}B_{j}$, there is some finite sequence of edges from $z$ to $v_{0},$ since $J$ is weakly connected.  Let $f$ be the first edge in this sequence such that one endpoint is in $\bigcup_{j=0}^{\ell-1}B_{j}$ and the other is not.  Let $x:=\sigma_{J}(f)$ and $y:=\tau_{J}(f)$.  If $x\not\in\bigcup_{j=0}^{\ell-1}B_{j}$ and $y\in B_{j}$ for some $0\leq j\leq\ell-1$, there is $g\in\edges_{J}\left(y,v_{j+1\pmod\ell}\right)$.  Then, the walk
\[\begin{array}{ccc}
w_{k}:=\left\{\begin{array}{cc}
x,	&	k=1,\\
y,	&	k=2,\\
v_{j+k-2\pmod\ell},	&	3\leq k\leq n,\\
\end{array}\right.
&	\textrm{and}	&
f_{k}:=\left\{\begin{array}{cc}
f,	&	k=1,\\
g,	&	k=2,\\
e_{j+k-2},	&	1\leq k\leq n-1,\\
\end{array}\right.
\end{array}\]
forces an edge from $v_{j-2\pmod\ell}$ to $x$, meaning that $x\in B_{j-2\pmod\ell}$.  A similar contradiction occurs if the roles of $x$ and $y$ are reversed.  Therefore, $V(J)=\bigcup_{j=0}^{\ell-1}B_{j}$, meaning $J$ is a blow-up of $C_{\ell}$.

\end{proof}

\section{Concluding Remarks}

As seen with the example of $\phi_{n}$, injectivity relative to a single embedding led to a propagation of structure, characterizing an entire class of quivers.  A logical next step would be to consider different embeddings and to find new classes of examples that might be characterized in this way.  Dually, the same question could be asked with regard to projectivity relative to a covering map, though this seems more mysterious due to Example \ref{projectivefail}.

\begin{figure}
\[\xymatrix{
A\ar@/^/[rr]^{\alpha}\ar@/_/[rr]_{\beta}	&	&		B\\
}\]
\caption{The Category $\cat{K}$}
\label{equalizer}
\end{figure}

Further, different graph-like structures could be considered.  Recall that  the category of quivers can be considered as a either a functor category \cite[Definition I.6.15]{joyofcats} from the category $\cat{K}$ drawn in Figure \ref{equalizer} to $\Set$, or as a comma category \cite[Exercise I.3K]{joyofcats} between the identity functor and diagonal functor on $\Set$.  In either case, replacing $\Set$ with the category of topological spaces creates the category of topological, or continuous, graphs considered in \cite[Definition 1.1]{deaconu2000} and \cite[Definition 2.1]{Li2012}.  Similarly, replacing $\Set$ with a category of weighted, or normed, sets in \cite[\S 2]{grilliette1} creates a respective category of weighted graphs.  Moreover, in the comma category setting, replacing the diagonal functor with the covariant power set functor gives rise to the category of hypergraphs considered in \cite[\S 2]{dorfler1980}.  Thus, the methods and results of this paper could be adapted to each of these settings.

\end{document}